\documentclass[12pt,oneside]{amsart}
\usepackage{amsmath,latexsym,amssymb,xcolor,graphicx}

\newtheorem{Alg}{Algorithm}
\newtheorem{Lem}{Lemma}
\newtheorem{Prop}{Proposition}

\newtheorem*{Rem}{Remark} 
\newtheorem*{Main}{Main Theorem}

\definecolor{mygreen}{rgb}{0,.5,0}
\colorlet{mypurple}{-green!40!yellow}

\newcommand{\bce}{\begin{center}}
\newcommand{\ece}{\end{center}}

\newcommand{\mat}[4]{\left( \begin{array}{cc}
{#1} & {#2} \\
{#3} & {#4} \end{array} \right) }
\newcommand{\alp}{{\alpha}}

\newcommand{\del}{{\delta}}

\newcommand{\eps}{\epsilon}

\newcommand{\modsurf}{\ensuremath{\Gamma(1) \backslash {\mathcal H}}}

\newcommand{\modu}{\ensuremath{\Gamma(1)}}

\renewcommand{\phi}{\varphi}
\renewcommand{\theta}{\vartheta}
\renewcommand{\Xi}{\varXi}

\renewcommand{\rm}{\normalshape}

\usepackage[pdftex]{hyperref}

\begin{document}
\title{Individual Closed Horocyclic Orbits on the Modular Surface}
\author{Marvin Knopp}
\address{Temple University\\Philadelphia, PA
19022}
\email{knopp@temple.edu}
\author{Mark Sheingorn}
\address{220 Shaker Blvd\\Enfield, NH 03748}
\email{marksh@alum.dartmouth.org}
\urladdr{http://www.panix.com/~marksh}
\date{\today}

\begin{abstract} We  track the trajectories of  individual horocycles on  the modular surface.     Our tracking is constructive,  and we thus   \emph{effectively} establish topological transitivity and even line-transitivity for the horocyclic flow.     We also describe homotopy class jumps that occur under continuous deformation of horocycles.
\end{abstract}

\maketitle

\section{Introduction}\label{int}

\subsection{Main Result}

This  project began when the first author asked the second for a description of  the path of the projection to the modular surface,  \modsurf  \ of the horocycle $I_{n} := \{ z = x + i/n \: | \: x \in [0,1], n > 0\}.$  
The standard fundamental region, denoted  \emph{SFR}, for  the full modular group \modu, is the set of $z$ satisfying $\{ z = x + iy \; | \; 1/2 \leq x < 1/2, \text{ and  } |z| > 1,  \text{ together with the portion of  } |z| = 1 \text { with } 0 \leq \Re{z} \leq 1/2\}$.  So defined the $\text{SFR}$ contains one elliptic fixed point of order three (\emph{efp3}),  denoted as usual  $\rho = 1/2 + (i \sqrt{3})/2$   and  one  elliptic fixed of order two (\emph{efp2}), the point $i$.   The side pairings are accomplished with $T := \mat{1}{1}{0}{1}$ and $S := \mat{0}{1}{-1}{0}$. 

We examine horocycles  $I_\alp := \{ z = x + i\alp \: | \: x \in [0,1], \alp > 0\}$ whose common point (see the Nielsen-Fenchel manuscript \cite{wF}) is $\infty$.  The interesting case is  that of small $\alp$.    Each of these horocycles projects to a closed curve,  as the horizontal line between $i\alp$ and $1 + i\alp = T(\alp)$ is a fundamental segment. That arc has hyperbolic length $1/\alp$. However, the hyperbolic distance between $i\alp$ and $1 + i\alp$ is about $\log{(1/\alp)}$.  In that sense the paths we study are long.

The goal of this work is to obtain explicit results about the dynamics of horocycles.

\begin{Main}\label{Main}
 Let $H$ be (one of the two) horocycles on \modsurf \ passing thru $z \in SFR$ at angle $\Psi$ to the horizontal. Let  $\eps > 0 $ be arbitrary. Then for any angle $\Theta$ it is possible to explicitly construct an element $A = A_\eps \in \modu$ with $A(H)$ making an angle with the horizontal within $\eps$ of  $\Theta$ at $A(z)$. 
 \end{Main}

\begin{Rem}  Since the SFR is convex, this recurrence  is enough to establish topological transitivity and even line-transitivity for the horocyclic flow \underline{effectively}. (Compare \cite{sD1}.)  It is far short of ergodicity, however.
\end{Rem}
\bigskip

\subsection{Contents of the paper}

We study closed orbits --- these are horocycle boundaries where the real {\em anchor} of the horocycle is rational. Of course, each such horocycle is \modu-equivalent to a horizontal horocycle.

The introductory geometry and summary or prior work on the horocyclic flow appears in sections ~\ref{int}  thru ~\ref{target}. (This part of the paper includes subsection \ref{Versus} which indicates that, in contrast to closed geodesics, the paths of closed horocycles thru SFR and the words arising from the resulting products of side-pairings of SFR are not closely related.)   Section ~\ref{param} provides a parametrization useful in the fine study of closed horocycles.  That study appears in sections ~\ref{different} and ~\ref{ruling} which offer Algorithms \ref{Alg1} and \ref{Alg2}, respectievly. Either of these algorithms is sufficient to establishing Theorem \ref{Main}, and with it (effective)  line-transitivity of the horocyclic flow on \modsurf.

 Section \ref{homo} examines the homotopy classes  on \modsurf \ of lifts of the closed curves $I_\alp$ as $\alp \to 0.$   
Words of caution:  These curves are not geodesics.  Even though \modsurf, like all Riemann surfaces, has constant negative curvature, there is often no shortest member of a homotopy class.  Here,  as usual in point set topology,   two curves $f(t)$ and $g(t), \: t \in [0,1]$ on \modsurf \ are homotopic if there is  a continuous function $H$  from  $[0,1]^2$ to \modsurf \ with $H(t,0) = f(t)$ and $H(t,1) = g(t)$.   Also, even though \modsurf \ has genus zero there are homotopically non-trivial closed geodesics by virtue of  the cusp, the efp2, and efp3.  Indeed for any hyperbolic element in \modu,  the {$h$-line or hyperbolic line in upper half plane connecting  its fixed points will project to  a non-trivial closed  geodesic on \modsurf.

 The main phenomenon established in section \ref{homo} is that the homotopy classes only change as the we encounter elliptic fixed points in descending. This extends  to all of $\alp >0$.  Note that there can no such stability for, say, the set of vertical geodesics with real part in the interval $\mathcal{I} := [\alp - \eps, \alp + \eps]$, where $\eps$ is much  smaller than $\alp$.  This is because the trajectories are given by the continued fraction expansion of the the real end, i.e., the {\em foot} of the vertical;   once the partial denominators in the expansions  are greater than $1/\sqrt{2\eps}$, the presence of the foot in $\mathcal{I}$ provides no further information about later partial denominators, and therefore no information about the path.

Two classical (\cite{bS}, pp. 10-12) computational lemmata which are useful in determining said homotopy classes  appear in subsection ~\ref{lemmata}. These are followed by all the homotopy classes  down to $\alp = 1/(2\sqrt{3})$ in subsection ~\ref{classes}. \\

 \subsection{Comments on Previous Work}\label{previous} 
 
The study of the geodesic orbits on arithmetic surfaces is extensive (see the (somewhat dated) bibliography in \cite{mS2} as a starting point).  In fact, complete knowledge of the vertical geodesic orbits\footnote{These are often called \emph{up the cusp} geodesic orbits.} is equivalent to complete knowledge of horocyclic orbits as the two are orthogonal and the action of $PSL(2, \mathbb{R})$ on such surfaces is conformal. However,  as explained below in section \ref{orthtrans}, the correspondence is not well suited to effective computation.  We offer two ways to do better, both involving continued fractions.\\

Though the efforts on geodesy dominate, there has been a significant previous work on the horocyclic orbits.  Among the mathematicians involved in this effort, which is usually said to have begun with Hedlund in 1936 \cite{gH},  are Hejhal \cite{dH}, Minsky and Weiss \cite{yM}, Karp and Peyerimhoff \cite{lK}, Masur \cite{hM} and Ratner \cite{mR}.  
 
 A simply {\em stated} sample of this work is  Masur's 1981  paper   \cite{hM} proving that the horocyclic flow is transitive --- meaning that any two open sets in the  unit tangent bundle will eventually intersect at some time. 
 
The 1984  paper \cite{sD} of Dani and Smillie showed that the horocyclic flow is uniformly distributed.  In this work there is almost always a finiteness condition on the group/surface.  Infinite volume  features the transience of almost all paths (into collars of free sides, for example). 

In 1981 Sarnak \cite{pS} treated the asymptotics of closed horocycles on finite volume fuchsian groups with cusps.  The limiting behavior is the concern.

The 1986 paper of Dani \cite{sD1} takes treats dense horocyclic orbits in a far more general setting than  the modular surface considered here.

 A last example, perhaps the closest in spirit to this paper, is that of the 2000 paper of Hejhal \cite{dH}.  This paper treats closed horocycles as we do here, but for more general fuchsian groups.  It infers ergodic properties from test-function integral averages.
 
 The modular group as a separate object is not treated and discussion of individual paths or, equivalently, elements of the unit tangent bundle are either not treated at all or not in the foreground.

\section{Basic Closed Horocycle Geometry}

 \subsection{Horizontal (Closed) Horocycles}\label{hch}
 
Recall  $I_{1/n} := \{z\; | \; z = x+i/n, \; x \in [0,1]$. (In section \ref{homo} we will consider more general $I_\alp$.) $I_{1/n}$  may be lifted to the standard fundamental region (SFR)  for \modu, of course.  The first aim here is to give the sequence of \modu \ elements effecting this lift.  We will list them going from left to right ---  as the anchors increase. 

Consider the  upper half plane closed horocycle $ \mathcal{H}  = \{z = x+iy \; |  \; y \geq 1\}$.   This may be mapped by the   \modu \ transformation $\mat{p}{*}{q}{*}$ to the horocycle anchored at $p/q.$ Of course this is not unique, but we will choose the stars to be the numerator and denominator of the next largest element of the Farey sequence (of level $[\sqrt{n}]$ --- see below).  Below these are denoted $p'$ and $q'$.  (These images of  $\mathcal{H}$ are disjoint --- apart from possible tangency --- and they have Euclidean radius $1/(2q^2)$.  They are called {\em Ford circles} and we denote them $\mathcal{F}_{p/q}$.)

We shall be concerned (only) with the Ford circles with $0 \leq p/q \leq 1$ and $q \leq \sqrt{n}$.  This is exact set of horocycles that  intersect $I_{1/n}$.  We illustrate in Figure 1 the circles for $q=5$.

 \begin{figure}[h]\label{Fi:Ford}
 \centering \includegraphics[scale=.4]{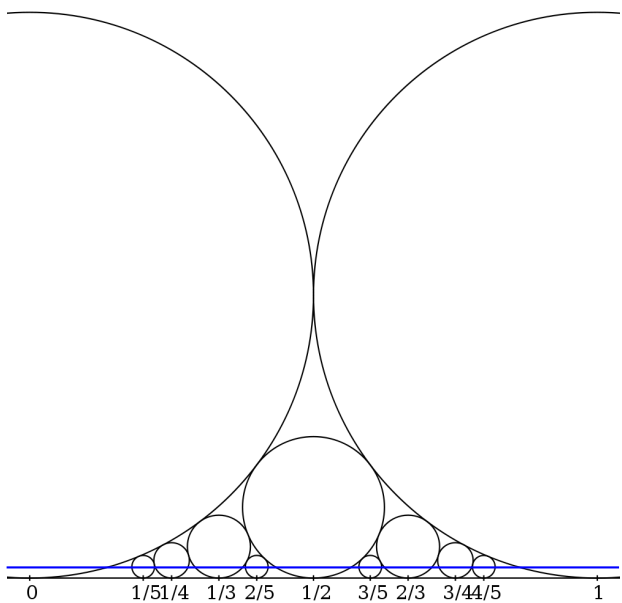}
\caption{Ford Circles for $ q \leq 5$ and {\color{blue} $I_{1/n}$}}
\end{figure}

 (If $n$ is a square, then $I_{1/n}$  is tangent to the smallest Ford circles in our considered set.  If not, it enters and leaves all these Ford circles.)

Going from left to right, the order in which  the anchors of the circles which are encountered by $I_{1/n}$ is the Farey sequence at level $q$, the largest denominator.  Standard Farey asymptotics (\cite{wL}, Vol. 1, p. 120) indicate that there are about $(3q^2)/\pi^2 \approx (3n)/\pi^2$ circles in our set. 
The curve $I_{1/n}$ is horizontal, in particular it enters and leaves each Ford circle at the same height; therefore,  the closest its lift to the SFR  comes to $\infty$ (that is,  the furthest it gets from the boundary of this horocycle) is in the middle, at $x = p/q$.  The height of the lift of that point is  $n/q^2$.  This is greater than  or equal to $1$;  the value is without upper bound,  as $q$ might be small  and $n$ large. The anchor of the lifted $I_{1/n}$ is $-q'/q.$    Knowing that and maximum height of the lifted $I_{1/n}$  is enough to give the $x$-coordinates of the entry points of this lift to $\mathcal{H}$. These are $-q'/q \pm \sqrt{n^2/q^4 -1}$. 

After the lift reaches the boundary of $\mathcal{H}$, it is clearly necessary to apply $S$ in order to remain in the SFR.  This is not readily viewable in the Ford circles; however the next Ford circle we enter gives us the (absolute value of the) next exponent on $T$.

In the geodesic setting, where the focus is on primitive hyperbolic elements, keeping track of the side incidences gives a word in $S$ and $T$ for the primitive element. We shall see that this is not the case for horocycles.
(Here a  hyperbolic element in \modu \ is {\em primitive} if it is not the power of another hyperbolic element in \modu.)

We now focus on the portion of $I_{1/n}$ in each Ford circle at $p/q$. Call this segment  $I_{1/n, \, p/q}$. We want to lift  $I_{1/n,\, p/q}$ into the horocycle $\mathcal{H}$.  After a direct computation, one finds:

\begin{Lem}\label{Ford lift} The lift of $I_{1/n, \, p/q}$, the portion of $I_{1/n}$ in $\mathcal{F}_{p/q}$,  to SFR is effected by $A := \mat{-q'}{p'}{q}{-p}$. The apex is $\frac{-q'}{q} +\frac{in}{q^2} = A(\frac{p}{q} + \frac{i}{n})$.

\vspace{.1in}

Also, $A(\frac{p}{q} + \frac{i}{2q^2}) = \frac{-q'}{q} + 2i$, $A(\frac{p}{q} + \frac{i}{q^2}) = \frac{-q'}{q} + i$,  and $A(\mathcal{F}_{p/q}) = I_1.$
   \label{hgt}
\end{Lem}

The fact that the real parts of the various $A$-images do not depend on $n$ (which must be greater than $q^2$ for this arc to be in SFR) means that all these lifts are (large) horocycle arcs based at the same point, $-q'/q$.  This is a vantage point where $q$ is fixed and $n \to \infty$ (and $I_{1/n}$ approaches the real axis).  Note also that the continued fraction of $q'/q$ is easily gotten from that of $p/q$.

 It is of interest that the hyperbolic distance between $i/n$ and $1 + i/n$ is $\log(\frac{n^2 + 2 + \sqrt{n^2 + 4}}{2}) \backsim 2\log{n}$. However, the arc length of of the horizontal horocycle between these two points is far larger; it is exactly $n$.

Also, note that if the horocycle encounters an efp2, the situation is different than for a geodesic, which reverses direction along the same h-line.  For the horocycle, we proceed along a different lift.  The salient example is $I_1$ hitting $i$.  An application of $S$ then has the path proceeding along the horocycle with anchor 0 and euclidean center $i/2$ --- that is, $\mathcal{F}_{0/1}$.

\subsection{Word Versus Path: The case of $I_{1/(n^2+1)} $} \label{Versus}

This section  considers the relationship of the incidence sequence in SFR of the lifts of $I_{1/n}$ and the word obtained by multiplying out the side pairing. In the case of primitive hyperbolics and the lifts of their axes, these are equivalent.  Here we see that this is far from the case for $I_{1/n}$.  (To ease computation we use $I_{1/(n^2 + 1)}$ to illustrate.  There are no new issues arising for general $n$, however.)

\subsubsection{The Exponent of the Translation $T := z \mapsto z+1$}

The aim of this subsection is to prove

\begin{Prop}\label{path}
 Lift $I_{1/(n^2 +1), \, p/q}$  to $\{ y \geq 1\}$ with $A$ as in Lemma \ref{Ford lift} and proceed along that path from left to right.  Then the powers of $T$ required to remain in SFR switches sign twice, starting positive, then negative, then back to positive.  Summing these exponents of$T$ gives  
$-(2n)/q  \pm 1$.  Thus the exponent is always negative.  \end{Prop}

From Lemma \ref{Ford lift}, we see that $A$ lifts $\mathcal{F}_{p/q}$ above $y=1$.  This is just the  horocycle $\mathcal{F}_{-q'/q}$. Of course  $-q'/q < 0$. We shift  $\mathcal{F}_{-q'/q}$ into SFR by adding $\lfloor q'/q \rfloor$.   Denote this again as $\mathcal{G}_{-q'/q}$.  Pass the line $y = 1$ thru $\mathcal{G}_{-q'/q}$.  The intersection points are at $\mp n + \lfloor q'/q \rfloor + i$
  
Start at the left intersection of $y = 1$ with $\mathcal{G}_{-q'/q}$ and travel over the latter horocycle until we reach the right intersection.  This whole path lies in $\{ y \geq 1\}$ and as such successive arcs maybe be translated to SFR using only $T^{\pm 1}$. At the outset we go backwards (decreasing real part).  So at this point we use a sequence of $T^1$'s to remain in SFR.  That obtains until we reach the height of the center (euclidean, $\mathcal{G}_{-q'/q}$ isn't a hyperbolic circle) which is $(n^2 +1)/(2q^2)$.  Then we start going forwards (increasing real part).  This requires application of $T^{-1}$. 

We continue with $T^{-1}$ until we return to the height of center when we start going backwards again, requiring $T^1$.  
Thus \emph{in toto} negative direction always prevails since we  have the whole top of $\mathcal{G}_{-q'/q}$ but not the whole bottom.  The final exponent of $T$ is obtained  (differences of the) real parts of four points on the horocycle: $\alp, \beta, \gamma, \delta$.  These are, in order, left intersection with $y = 1$, left intersection with horizontal diameter, right intersection with horizontal diameter,  right intersection with $y = 1$. Since we only care when we leave SFR, not how long we spent there, the real parts are all we need.

Indeed since only the differences of the real parts matter, we might  as well move the anchor to 0.  
This gives $\alp, \delta = \pm \sqrt{(n^2+1)/(q^2) -1} \sim \pm n/q$ and $\beta, \gamma = \pm (n^2+1)/(2q^2)$.
The $T$ exponent tally is then $(n^2 +1)/q^2$ negative ones and $(n^2 +1)/q^2 - (2n)/q$ positive ones.  The exponent in total is  within 1 of $-(2n)/q$.

After leaving $\delta$ heading towards anchor of  $\mathcal{G}_{-q'/q}$ we are below $y = 1$ and thus next an involution will be applied to remain in SFR. For an integer $N$, define the \modu \ involution fixing $N+i$

\begin{equation} \label{E_N} E_N := \mat{N}{-1-N^2}{1}{-N}
\end{equation}

\subsection{Applying $E_N$}  

The aim of this subsection is to offer an example where continuing Lemma \ref{path}'s  lift of  $I_{1/(n^2 +1), \, p/q}$ below $I_1$ immediately leads to two consecutive applications of the same involution.

 In this subsection we let  $q = 1.$  Then Ford circle at $0 = 0/1$ has center $i/2$.  Now $E_0(I_{1/(n^2+1)}) := \mathcal{G}$ is the horocycle anchored at $0$ with euclidean radius $(n^2 + 1)/2$.  $\mathcal{G}$ intersects $\{ y = 1\}$ at  $ \alp, \delta = \mp n + i$, resp.  

Note that we have that $\alp, \delta$ are themselves efp2s, simplifying matters considerably. That is, the involution to be applied has  $N = \mp n$.  By (\ref{E_N}) we have that the anchor of $E_n(\mathcal{G}) = n+ 1/n = (n^2 + 1)/n$. This horocycle is $\mathcal{G}_{(n^2 + 1)/n}$. 

The euclidean center of  $\mathcal{G}_{(n^2 + 1)/n}$ is obtained by using the fact the $E_n(\delta) = \delta = n + i$ is on this circle.  Said radius is $(n^2 + 1)/(2n^2) > 1/2$ showing that we do enter SFR.

The next two side pairings   call in to question the whole notion of a visceral connection between the curve $y = 1/(n^2 +1)$ and the word given by the side incidences. To wit, when $\mathcal{G}$ leaves the SFR at $n+i$ we apply $E_n$ to get back in, as we just saw. The next side incidence is at $n + 2/n +i$.  When $n \geq 3$ this incidence is on the side paired by $E_n$.  But we just used that and it is an involution!  So in our word, the product of the two  $E_n$'s disappears.  Yet we need it for the geometry as $n+i$ is on $\mathcal{G}$. Figure 2 illustrates.

\begin{figure}[h]\label{Fi:ks}
\centering\includegraphics[scale=1.3]{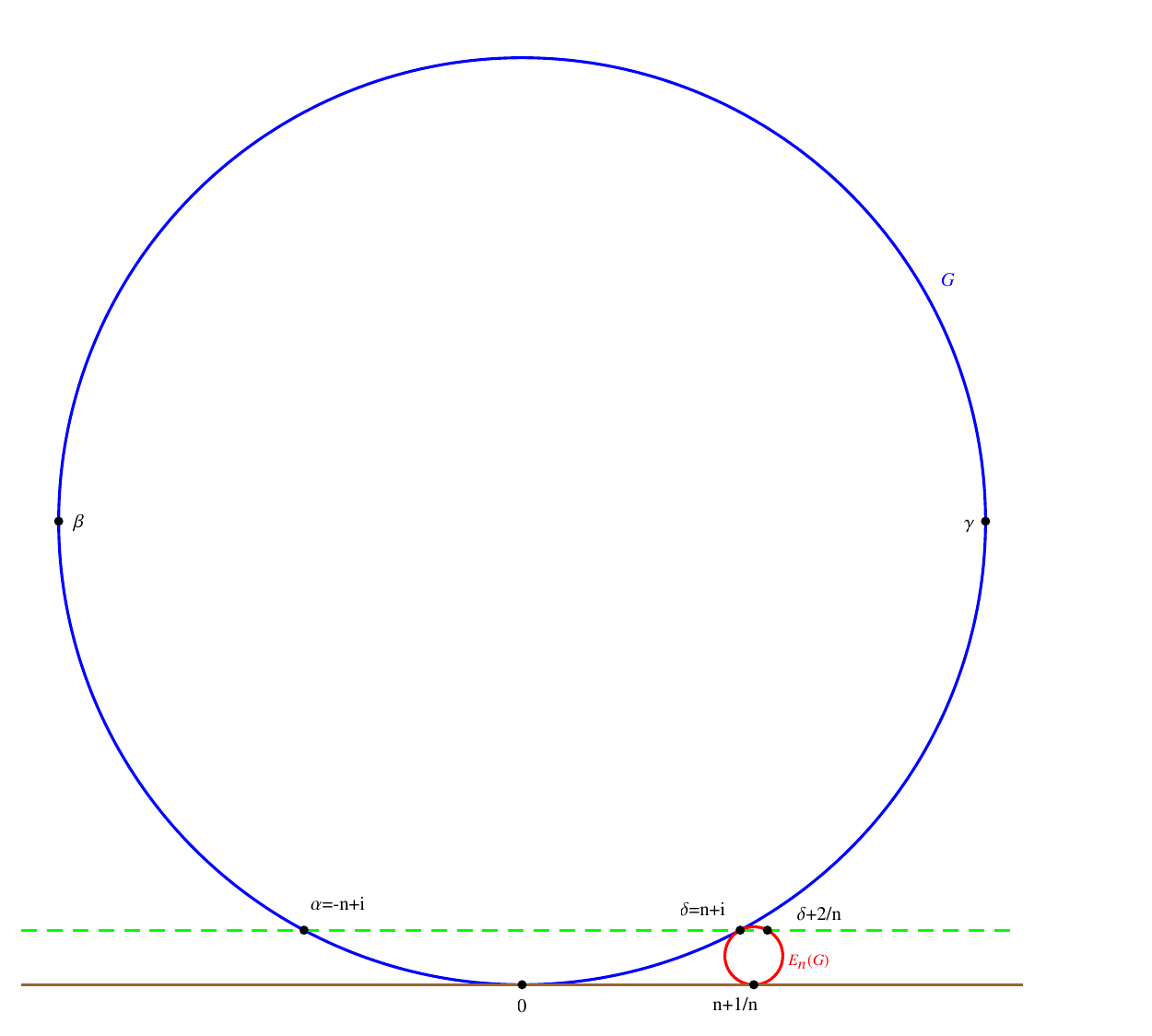}
\caption{$\mathcal{G}$ and $E_n(\mathcal{G})$}
\end{figure}

$\mathcal{G}$ and $E_n(\mathcal{G})$ are in fact tangent at $\delta$.  If they were not, the two horocycles would intersect again and that second intersection would also be a fixed point of $E_n$.  But $E_n$ has only one fixed point in $\mathcal{H}.$

\subsection{Lift Direction}

Next consider the action of $A^{-1}$ on $\mathcal{F}_{p/q}$, where $A$ is from Lemma \ref{hgt}.   The proof of the following Lemma is a direct computation.

\begin{Lem}  With $A$ as in Lemma \ref{Ford lift}
let $\beta > 0$ and $\alp \geq 2$. Then  

\begin{gather*}
A(p/q) = \infty \\
A(\infty) = -q'/q \\
A(\frac{p}{q}\pm \frac{1}{\beta q^2} + \frac{i}{\alp q^2}) = \frac{-q'}{q}  \mp \frac {\alp^2 \beta}{\alp^2 + \beta^2} + i \frac{\alp \beta^2}{\alp^2 + \beta^2}.
\end{gather*}
\end{Lem}

The key computation, of course,  is the last one. Note that if the imaginary part of the right hand side is at least one, which happens when

\[ \beta \geq \frac{\alp}{\sqrt{\alp -1}}, \]

then the image is in $y \geq 1$ so that the pre-image is in $\mathcal{F}_{p/q}$.  Next,  $\beta > 0$ so the expression in the real part involving $\alp$ and $\beta$ (without the $\mp$) is positive.  This means if  we proceed thru  $\mathcal{F}_{p/q}$ moving from real part smaller than $p/q$ along $I_{1/n}$ to real part greater than $p/q$, we proceed in the opposite direction along the lifted horocycle anchored at $-q'/q$.  (It is only necessary to check the endpoints on the boundary of $\mathcal{F}_{p/q}$.) This is independent  of $p$ and $q$!

 \section{Orthogonality Implies Transitivity}\label{orthtrans}
 
This section uses the fact  that horizontal horocycles are orthogonal to vertical geodesics to establish line-transitivity for the horocyclic flow, albeit ineffectively.

 The horocycles $I_\alp := \{y = \alp \}$ are orthogonal to all geodesics going up the cusp (vertical h-lines). Consider the vertical ending in $x \in [0,1]$.  We know that for almost all $x$, $CF(x)$ contains every finite sequence of positive integers.  Therefore, the usual continued fraction based geodesic tracking mechanism described in \cite{mS2}, shows that for such $x$ the lift of $ J_x := \{z | z = x + iy, y > 0\}$ to \modsurf \ passes near every point in nearly every direction. (This originates in the work of E.  Artin \cite{eA} and C. Bursteyn \cite{cB}.)
 
 \begin{Prop} The pencil of horocycle circles  $I_\alp$, when lifted to \modsurf, is dense in the unit tangent bundle.
 \end{Prop}
 
 \begin{proof}
The pencil of horocycle circles  $I_\alp$ is orthogonal to each h-line $J_x$.  Choose $x \in \mathbb{R}$ so that the lift of $J_x$ to \modsurf \ in dense in t he unit tanget bundle.  Conformality then requires that the lifts of infinitesimal pieces of the horocycles $I_\alp$ hitting $J_x$   also pass near each point of the unit tangent bundle. 
\end{proof}
 
In a very limited way this is effective --- given an appropriate $x$, and point in the unit tangent bundle,  and an $\eps$, a finite sequence of positive integers is determined. Locating it   in $CF(x)$ to find a point $x + i \alp$ where $I_\alp$ will be within $\eps$ of  the target in the unit tangent bundle. 

The difficulty is the one shared by the study of normal numbers --- though almost all $x$ are appropriate, it is not possible to tell whether a given real number is or not.  Rationals are not, nor are quadratic irrationalities, but is $\pi$ or $\sqrt[3]{2}$?  No one knows.  Appropriate $x$ can be artificially constructed, however.

 \section{Horocycle Parameters --- Construction and Targeting}\label{target}
 
  \subsection{Constructing Horocycles} \label{cons}
In this subsection we investigate the horocyclic flow along $I_{1/n}$.  
 
 Looking just at $I_{1/n}$,  we immediately see that there are huge differences between this and the geodesic flow.  All our horocycles are closed.  Almost all geodesics are not.  As such, while almost all individual geodesics can and do produce transitive (recurrent)  behavior; no single horizontal horocycle can.
 
Taking the  set of all $I_{1/n}$ lifted to SFR --- or even a countable set of such lifts like $((ip)/q, 1 + (ip)/q)$ --- might yield a transitive set in the unit tangent bundle. Each point/direction in the SFR might have a nearby approximation in the countable set of lifts. It also might be possible to study return times for this flow.
 
  Note also that  the relevant Ford circles remain the same for any horizontal  horocycle at height $1/y$ where $n \leq y \leq n+1$. 
  
To establish transitivity it would be sufficient to establish the directional approximation for each point along, say,  the h-line $\Re{z} = 0$.  The reason for this is that, given another SFR point/direction $(\tau, \Psi)$, we horocyclically follow that back to $\Re{z} = 0$, obtain an approximation to the new direction and then invert the process to get an  approximation at $(\tau, \Psi)$. Since the flow lines have no bifurcations, this is valid.  Also, because the SFR has just four sides, we reach a point on $\Re{z} = 0$ in short order. 
  
The point/direction pairs on $\Re{z} = 0$ have two parameters:  the point has an imaginary part $h$ with  $1 \leq h < \infty$.    And the angle $\Psi$ produced with $y = h$ may be chosen in  $0 \leq \Psi < 2 \pi$.  We could have an upper bound of $\pi$ if we were indifferent to the direction of progress on the horocycle. Also, since $y=h$ is horizontal, this is the standard slope.

Given $(h, \Psi)$, which horocycle goes thru $ih$ in direction $\Psi$? From this point forward we will use  interchangeably $x + iy$ and $(x,y)$, each with with $y > 0$, to denote a point in the upper half plane.  What we desire is the horocycle's real anchor point.  Clearly this is equivalent to knowing the euclidean center of the horocycle in terms of $h$ and $\Psi$. Call this center $(c,d)$ so that the horocycle is $(x-c)^2 + (y-d)^2 = d^2$.  (The radius is determined by the requirement of real-axis tangency.)  Since $(0,h)$ is on this circle, we have

\begin{equation} \label{horocycle eq}
c^2 + (h-d)^2 = d^2
\end{equation}

The direction of the tangent to the circle at $(0, h)$ satisfies

\begin{equation} \label{slope eq}
\frac{c}{h-d} = \tan{\Psi}
\end{equation}

N.B. : If one examines the configuration of the horocycle at the point of tangency $(0,  h)$, there are two antipodal possibilities --- providing two disjoint horocycles. In general, one has  $c < 0$; the other has $c > 0$.  At $c = 0$ the tangent is horizontal so one horocycle is anchored at $\infty$ and the other at $0$.  In this latter case $d = h/2$ then.

A straightforward computation yields the following

\begin{Lem}\label{horo cent}

Let $(c,d)$ be the euclidean center of the horocycle thru $(0, h)$ in direction $\Psi$. Then:

\begin{equation} \label{d eq}
d = 
\begin{cases}\frac{h}{1+ \cos{\Psi} }, &\text{if $c > 0$}\\[.05in]
h/2 \ \text{or N. A. as horizontal}, &\text{if $c=0$}\\[.05in]
\frac{h}{1- \cos{\Psi} }, &\text{if $c < 0$}
\end{cases}
\end{equation}

Using equation (\ref{slope eq}) we find:

\begin{equation}
c = (h-d) \tan (\Psi)
\end{equation}

Thus:
\begin{equation} \label{c eq}
c = 
\begin{cases}  h \frac{\sin{\Psi}}{\cos{\Psi} + 1 }, &\text{if $c > 0$}\\[.05in]
0, &\text{if  $c = 0$}\\[.05in]
 h \frac{\sin{\Psi}}{\cos{\Psi} -1 }, &\text{if $c < 0$}
\end{cases}
\end{equation}

\end{Lem}

One could map these horocycles to one with anchor at $\infty$ which  gives 

\[y = \frac{1}{2d}\]

This is too crude for our purposes --- it tells us only that there is a point on this horizontal horocycle that lifts to $ih$ and makes an angle $\Psi$ with $y=h$.  It doesn't tells which point (with $0 \leq x < 1$) it is.

\subsection{Targeting Points in the Unit Tangent Bundle.} \label{param} 


This subsection starts with a point in  the unit tangent bundle of a closed horocycle and produces another such point of the horocycle making a  given angle the horizontal. This is essential to the algorithms offered  in sections \ref{different} and \ref{ruling}.

We want to establish that the horocyclic flow recurs at any given angle at $\Im{z} \geq 1 $.  This means that starting on any horocycle through such a point we must show that we return to an arbitrarily close point in an arbitrary direction.   As in section \ref{target} we  parametrize our initial conditions as $(h, \Psi)$, meaning our point is $(0, h)$ and the horocycle tangent there has slope $\tan{\Psi}$.

\begin{figure}[h] \label{Fi:parameters}
\centering \includegraphics[scale=.5]{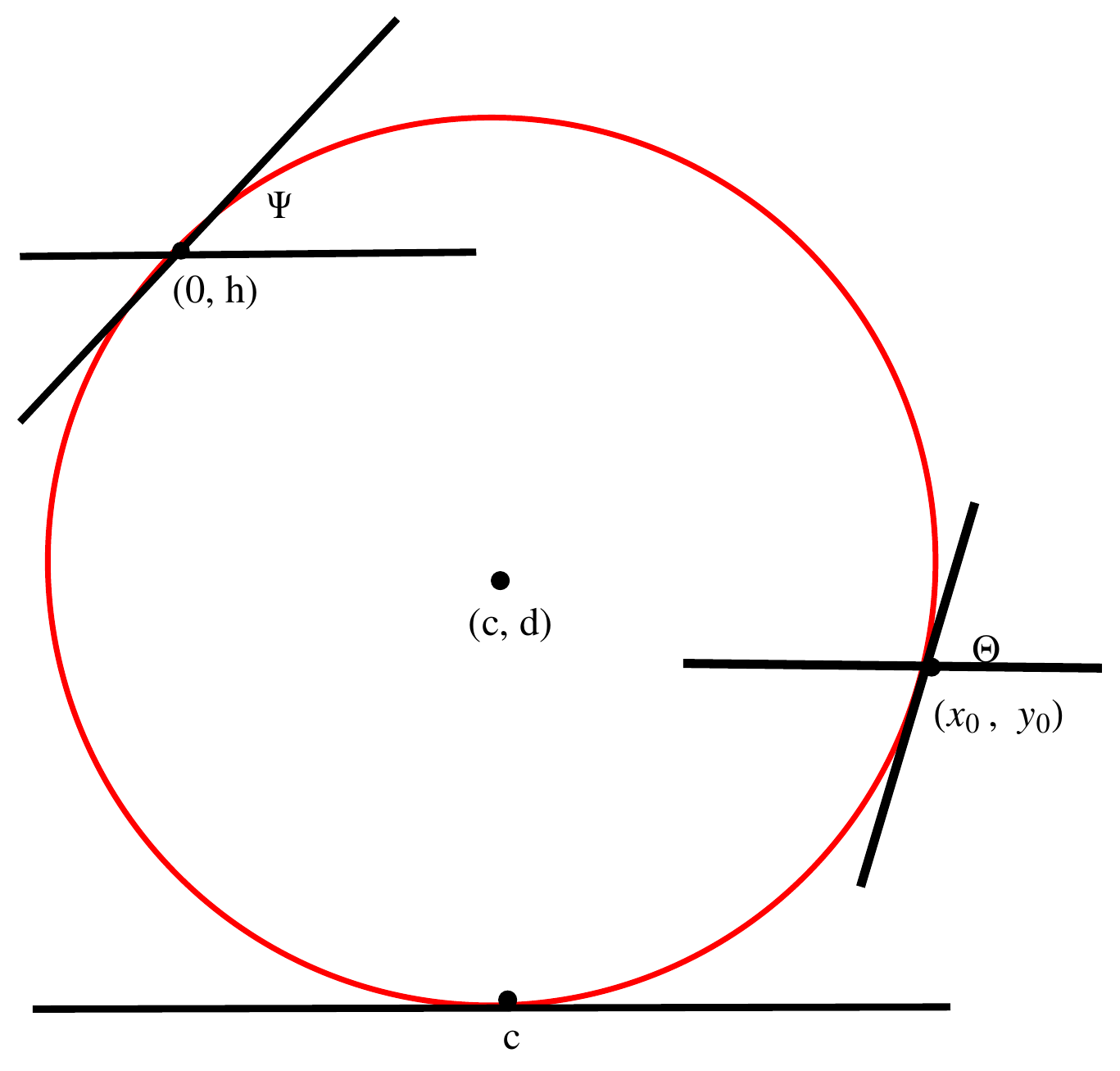}
\caption{ The $(h, \Psi)$ Parameters. Here $h>d.$}
\end{figure}

Note that $c$ and $d$ are determined in terms of $h$ and $\Psi$ by Lemma \ref{horo cent}. $\Theta$ is our target angle. We need  that $(x_0, y_0)$ is near a \modu \ image of $z = (0,h)$. As we've said, to expect that the horocycle contains $z$-image at which the horocycle makes an angle $\Theta$ with the horizontal is not generally possible.  These are horocycles are closed orbits.  However, the  flow is ergodic on the unit tangent bundle, so there is a nearby orbit --- nearby to $(h ,\Psi)$ --- that would have such a \modu \ image.

Using the notation of  Figure 3, we give explicit formuli for two such $(x_0, y_0)$  first in terms of $(c,d)$ and then in terms of $\Psi$:

\begin{Lem}\label{target pt} Consider the horocycle passing thru $(0, h)$ at angle $\Psi$.  Then the two points on this horocycle where an angle of $\Theta$ is made with the horizontal  are: 

\begin{gather}
x_0 = c \mp d \sin{\Theta}  \label{x_0 equation}\\
y_0 = d \pm d \cos{\Theta}  \label{y_0 equation} 
\end{gather}

Replacing $(c,d)$ dependence  with $\Psi$ dependence  gives (in the case of interest, $y_0$ small) :

 \begin{gather}
x_0 = \frac{2 h( \sin (\Psi)+ \sin (\Theta))+\cos (\Psi)+1}{2
   (\cos (\Psi)+1)}  \label{newx_0 equation}\\
y_0 = \frac{h(1- \cos (\Theta))}{\cos (\Psi)+1} \label{newy_0 equation} 
\end{gather}

\end{Lem}

 The issue is to ensure that there is such an image of $ih$ near $x_0 + iy_0$.   In the following two sections we do just that. The first approach in section \ref{different} seems more accurate --more precise output for fixed input. It involves the presence of $i$ on the boundary of SFR.  However,  it is less directly related to our parameters than the method of following section \ref{ruling}.
 
 Before we set out we note that we have not used \modu \ at all --- apart from fact that $ih$ is in the SFR provided that $h \geq 1$.  As Kurt Mahler noted, to make a hamburger you have to put in some meat. 
 
 In the following section we begin base point $i$.  This is only to simplify the resulting expressions. Later in the section we will indicate what (little) changes are necessary to pass to $ih$, $h\geq 1$ and then to the rest of the SFR.
 
 The goal of the next two sections is to offer two algorithms the Main Theorem.
   
 \section{First Algorithm: Transitivity point-wise.}\label{different}
 \subsection{Beginning with $i$}\label{base i}
 
  At the outset we use base point $i$ rather than any $ih$ in the SFR. This is only to simplify and clarify  the algorithm.  The extension to general $h \geq 1$ appears in subsection \ref{baseih}. 
  
  Start with the SFR and the horocycle $y=1$.  This passes through $i$ of course, making an angle of zero with unit circle $\mathcal{U}$, which contains the sides of the SFR paired by $S$.  It is easy to keep track the group actions on $\mathcal{U}$.   By conformality, the slope at $A(i)$ of $A(\mathcal{U})$ is the same as the slope at $A(i)$ of the $A(I_1)$. And since we are interested in the angle produced by $A(I_1)$ with an appropriate $I_{1/n}$, this is just what we want.
  
\emph{(Caution: In the remainder of this section $c$ and $d$ are matrix entries, not the coordinates of the center of some horocycle.)}

  Let $A = \mat{a}{b}{c}{d}$. then $A(i) = \frac{(bd + ac) + i}{c^2 + d^2}$, which means $n$ is the sum of two relatively prime squares.  That is equivalent to $n$ being square free with no prime factors congruent to $3$ modulo $4$.  (This is most of the content of Lemma \ref{lowi} below.)

We wish the examine the angles at which the horocycles $I_{1/n}$ hit images of $i \in A(I_1)$ . 
Let  $A(\mathcal{U}) = \mathcal{B}$. Then $\mathcal{B}$  connects $A(-1)$ to $A(1)$. That is, $\mathcal{B}$ is  an h-line of euclidean radius $\frac{1}{|d^2 - c^2|}$ and center $\frac{bd-ac}{d^2 - c^2}.$ The real point of the boundary which is the image under $A$ of SFR is $A(\infty) = a/c$. Now $A(i)$ is on $\mathcal{B}$, and $I_{1/n}$ hits $A(i)$, where $n = c^2 + d^2.$  This incidence cannot be at the apex as the denominators don't match.  (The apex might be another \modu \ image of $i$,  however.) Therefore the angle produced at this incidence is not zero.  Since we have the $\mathcal{B}$'s  center and radius and the  $x$-coordinate of $A(i),$ it is a simple matter to get the slope. This is:

\begin{equation} \label{different slope} 
\frac{\left(c^2+d^2\right) (b d-a
   c)}{d^2-c^2}-a c-b d = \frac{2cd}{c^2 - d^2}
 \end{equation}
 
 If we let $f(x) = (2x)/(1-x^2)$, then the last expression in (\ref{different slope}) is $f(d/c)$.  But $d/c = -A^{-1}(\infty)$.  For \modu \ the set of such $d/c$ is dense in $\mathbb{R}^+$.  This, together with the fact that $f$ maps $\mathbb{R}^+$ onto itself, establishes recurrence at least at $z = i$ --- we can get as close to any slope as we like.
 
 Next, we  consider the equation $f(x) = \tan{\Theta}$ and solve it for $x$, thinking of $x$ as $d/c$.  That gives
 
 \begin{equation} \label{d/c of theta}
 x = x(\Theta) =  d/c = -\cot{\Theta} \pm \csc{\Theta}
 \end{equation} 

Thus starting with a target of $\tan{\Theta}$, we expand $d/c$ in a continued fraction the RHS of equation (\ref{d/c of theta}).  The $n$'s we want are of the form $n = n_k = p_k^2 + q_k^2$, i.e. $p_k = c, \, q_k = d$.  The larger $k$ is, the closer we are to $(A_{n_k}(i), \Theta)$ in the unit tangent bundle. Specifically, we have here that the $z$ entry is a \modu-image of $i$ (so that coordinate is exact), and the angle will satisfy $|X_\Theta - p_k/q_k| \leq 1/(2q^2_k)$ by standard continued fraction bounds found in  \cite{aK}, p. 36.  

If the nearness to $\Theta$ itself is desired, note that $\Theta = \arctan{f(X_\Theta)}$ and upper and lower bounds for the effect of $f$ and $\arctan{x}$ are at hand.

Though the arriving at  the algorithm is somewhat involved, the algorithm itself is simply rendered. For simplicity, we state it for $h = 1$, i.e., point $i \in $ SFR. The passage to $h \geq 1$ is treated next in subsection \ref{baseih}.

\begin{Alg}\label{Alg1}
The goal is to find a  \modu \ element A and a positive integer $n$ such that $A(I_1)$ makes  an angle nearly $\Theta$ with $I_{1/n}$ at $A(i)$.   
\begin{itemize}

\item Input: $\Theta \in [0, 2\pi)$, the target angle and $\eps$, the error we will accept for the auxiliary parameter $X_\Theta$, which is an acceptable proxy for $\Theta$ or the slope $\tan{\Theta}$. 

\item  Step One:  Evaluate $K =  -\cot{\Theta} + \csc{\Theta}$.

\item Step  Two: Find the continued fraction of $K = [a_0; a_1, a_2, \ldots ] = $ CF(K).

\item Step Three: Choose $k$ such that $1/(2q_k^2) < \eps$. (This gives  the error in the approximation of $X_\Theta$, which as we've observed is equivalent to bounding the error in $\Theta$ or of the slope $\tan{\Theta}$. 

\item Output: $n = p_k^2 + q_k^2$, 
\[ A = \mat{p_{k-1}}{(-1)^k q_{k-1}}{ p_k}{(-1)^k q_k} \]

\end{itemize}

\end{Alg}

Here is an explicit example.  Take $\Theta = 33^{\circ}$ and $\eps = .000001$ .  Then $\tan{\Theta} = 0.6494075932 \ldots$ and $- \cot{\Theta} + \csc{\Theta} =  .2962134950 \ldots.$  

The continued fraction of .2962134950 is (exactly) \newline $[0; 3,2,1,1,1,15,1,14,1,7,1,1,1,3,1,20,2,1]$.  Here are  values of two consecutive  convergents: $[0; 3, 2, 1, 1, 1, 15, 1] = 133/449$ and  $[0; 3, 2, 1, 1, 1, 15, 1,14] = 1987/6708.$

Therefore the \modu \ matrix $A$ is 

\[ A = \mat{133}{449}{1987}{6708}. \]

(Aside: A choice has been made here; the choice we made in Algorithm \ref{Alg1}. We could have used the matrix 

 \[ \mat{133}{-449}{-1987}{6708} \]
 
  This would replace $x$ by $-x$ just below, while leaving the denominators of $x$ and $y$ unchanged. It would reverse the sign of the slope in equation (\ref{different slope}) and would reverse the sign of the real center (denoted $C$) of the axis of $A$, whose formula was $\frac{bd-ac}{d^2 - c^2}$.  This will be relevant below where the term $C-x$ appears would have its sign reversed if one makes this other choice.)

Also,

\begin{equation} A(i) := x + iy =   \frac{3276163}{48945433}+\frac{1}{48945433} i.
 \label{A(i)}
 \end{equation}

The real center $C$ of the h-line between $A(1)$ and $A(-1)$ is  $2747621/41049095.$ 
The slope of this h-line at $A(i)$  is $(C-x)/y.$

Substituting this $C$ and $x, y$ from equation (\ref{A(i)}) gives a slope of $-0.64940754$.  
Recall that the $\tan{33}^{\circ} = 0.64940759.$  We have agreement to 7 decimal places, modulo the sign.  We resolve this  by making the other choice alluded to in the Aside just above.  That reverses the sign of the slope, which is within $\eps$ of $\tan{33}^{\circ}$.

Finally, note that we have found $I_{1/48945433}$ is the horocycle with the approximating angle/image and this has height $2.043\ldots \times 10^{-8}$.

\subsection{Passing to $ih, h > 1$.}\label{baseih}

Here are the requisite changes this generalization requires:

\begin{itemize}

\item First, 

\[ A(ih) = \frac{(bd + ach^2) + i}{h^2c^2 + d^2}\].

\item Second, the center of $\mathcal{B}$, the h-line connecting $A(-h)$ with $A(h)$, is 
\[ \frac{A(-h) +A(h)}{2 } = \frac{bd - ach^2}{d^2 - c^2h^2}. \]

\item Third, the radius of $\mathcal{B}$ is 

\[ |\frac{A(-h) -A(h)}{2 }| = \frac{h}{|d^2 - c^2h^2|}. \]

\item Fourth, the slope of $\mathcal{B}$ at $x=  \frac{(bd + ach^2)}{h^2c^2 + d^2}$ is

\[ \frac{2hcd}{c^2h^2 - d^2}. \]

\item Fifth, the auxiliary function is now 

\[ f(x) = \frac{2hx}{h^2 - x^2} \]

where, \textbf{and this is critical}, $x = d/c$ as before. Continue with a specific $h$ and target $\Theta$ chosen and once again solve  
\[f(x) = \tan{\Theta} \]
Then expand the solution as a continued fraction as before, getting $p_k = c, \, q_k = d$ again, with $k$ controlling the accuracy to the target angle.  

\item Finally, The $n$ involved in $I_{1/n}$ is now the reciprocal of the height of $A(ih) = \frac{d^2 +c^2h^2}{h}$.  True it need not be an integer, but $I_\alp$ makes fine sense for any $\alp > 0$.

\end{itemize}

We also remark that there is no obstruction to running this algorithm starting at any base point in the SFR.  If this point were $z := x + ih$, we calculate $A(z), A(x-h)$ and $A(x+h)$, the latter two defining $\mathcal{B}$ and proceed as above.  The formulae are more involved, but the method works.

\section{Second Algorithm: Ruling the Horocycles}\label{ruling} We use the notation of section \ref{different}. 
.

 We want to examine the horocyclic flow in a small open set $\mathcal{H}$.    The motivation here is that we know the $A$-images of (the complete extensions of) those sides --- horizontal horocycles and vertical  h-lines. In effect we've placed rectangular (euclidean) $\mathcal{H}$

Let  $\mathcal{T}$ be the he set $\{z = x + iy\: | \: x \in (0,1/2], |z| > 1\}$. It is convenient for  this algorithm to restrict $z$ to the union of  $\mathcal{T}$, $R(\mathcal{T})$, where $R : z \to -\bar{z} +2$, and the portion of  the h-line $(x-1)^2 + y^2 = 1$ with real part in $(1/2, 1]$.  This  is a fundamental region for \modu, of course.  We shall denote it $\mathcal{F}^*$.  Every point in  $\mathcal{F}^*$ has real part in $(0,1]$ and absolute value greater than 1.

   Denote the SW rectangle corner $z$ and the NE corner $w$.  The the slope of the euclidean line between $A(z)$ and $A(w)$ is 

\[ \frac{\Im{A(z)} - \Im{A(w)}}{\Re{A(z)} - \Re{A(w)}} = \frac{\Im{(A(z) - A(w))}}{\Re{(A(z) - A(w))}}.\]

Next, note that 

\[ A(z) - A(w) = \frac{z-w}{(cz + d)(cw + d)}. \]

We're interested in the ratio of the imaginary to the real part of $A(z) - A(w)$.  This ratio is unaffected by multiplication by real functions. We exploit this as follows.

Note first that 

\[ \frac{ A(z) - A(w)}{z-w} \]

is nearly the derivative of (any) point inside the little box. That is, if our box has vertices

\[ z := x + iy; \,\,\,\ w := z(1 + \eps) \]  

we get, employing the derivative at $z$,  that

\begin{multline}\label{get real} 
 A(z) - A(w) \approx  (z-w) A'(z) =  \\ \frac{\eps z}{(cz+d)^2} = 
 \frac{\eps}{|cz + d|^4} (x(c^2 |z|^2 + d^2) + 2cd|z|^2 + i(d^2y -|z|^2c^2 y)
 \end{multline}

Note that the last fraction in (\ref{get real})  is real so that the ratio we seek is unaffected by it.  In fact, by taking $\eps = \eps(A, z)$ as small as we like, we can make the derivative approximation of $\Delta A/ \Delta z$ as close as we like.  That said, employing the derivative introduces a sort of error not present in Algorithm \ref{Alg1}. In the sequel, we may not  more partial denominators to get the same accuracy. 

Our ratio is then

\begin{equation}\label{Fdef}
 \frac{y -X^2y(x^2  + y^2)}{X^2 x\left(x^2+ y^2\right)+2X \left( x^2+
   y^2\right)+x} = F(\cdot, \cdot, X) = F(X)
   \end{equation}

\noindent with $X := c/d$ in the last.

(\ref{Fdef})  is a function of $X$ as $z = x+iy$ was chosen at the outset, subject only to the condition that $z \in \mathcal{F}^*$.   This condition implies  that the numerator and denominator of $F$ each have two real zeroes that are interspersed, like those of $(1-X^2)/(X^2 + 2X)$.    This interspersion forces the range of $F$  to be all of  $\mathbb{R}$.   Recall that $c/d = -1/A^{-1}(\infty)$ and that that set is dense in $\mathbb{R}$.

To conclude transitivity, take a point $z = z_0 := x_0 + iy_0 \in \mathcal{F}^*$. Take an  $\eps$-square box with SW corner  $z_0$ and a target angle $\Theta$.  Find $X_\Theta$ such that $F(x_0, y_0, X_\Theta) = \tan{\Theta}$.

Note a difference here from Algorithm \ref{Alg1}:  We haven't given a closed from expression for $X_\Theta$ like equation \ref{d/c of theta}. Nonetheless inverting $F$ is not difficult it is the quotient of two real quadratic polynomials.  This inversion permits us to go from $X_\Theta$ to $\tan{\Theta}$, from which we may pass to  $\Theta$ itself.

Last, find a $B \in \modu$ with entries satisfying $c/d$ near $X_\Theta$.  Then the action of $B$ will tilt the SW - NE  diagonal of the of the $\eps$-box to an angle against the horizontal of nearly $\Theta$. 

As with Algorithm \ref{Alg1}, arriving at  this algorithm is somewhat involved. But again,  this algorithm may be  is simply rendered.

\begin{Alg}\label{Alg2}
Starting with $z_0$ in the SFR and a target angle $\Theta$, the goal  is the find a $B \in  \modu$ which tilts the euclidean line between $z$ and $z(1+ \eps)$ so that the slope is within $\delta$  of $\tan{\Theta}$.

\begin{itemize}

\item Inputs: $z_0, \eps, \delta, \Theta$.

\item Step One: Find $X_\Theta$ such that $F(x_0, y_0, X_\Theta) = \tan{\Theta}$.

\item Step Two:  Take 

\[ k = \lceil - \frac{\log{\del}}{\log{2}} \rceil \]

 and find the first $k$ partial denominators in the continued fraction expansion of $X_\Theta.$ 

\item Output:

\[ B = \mat{p_{k-1}}{(-1)^k q_{k-1}}{ p_k}{(-1)^k q_k} \]

The choice of $k$ ensures that 

\begin{equation} \label{Alg2 bound}  | X_\Theta - p_k/q_k | < 1/(2q_k^2) \leq 1/2^k \leq \del .
\end{equation}

As indicated above, (\ref{Alg2 bound})  is equivalent to a similar bound on the difference of the slope of the euclidean line between $B(z_0)$ and $B((1+ \eps)z_0)$ and  $\tan{\Theta}$, or the angle said line makes with the horizontal and $\Theta$.

\end{itemize}

\end{Alg}

Here is an example:  Take $\Theta = 33^{\circ},    \, z_0 = 1 + i  \text{ and } \eps = \delta =  .0001.$  In subsesction \ref{base i} we  had $\tan{33}^{\circ} =  0.6494075932\ldots.$ Solving $F[1,1, X] = 0.6494075932$ gives $X = .1174485783$.  Now $CF[.1174485783] = [0; 8, 1, 1, 16, 1, 9, 2, 2, 1, 4, 1, 2, 1, 4, 1, 3]$.

Two consecutive convergents are $[0; 8, 1, 1, 16, 1, 9, 2] = 731/6224$ and $[0; 8, 1, 1, 16, 1, 9, 2, 2] = 1810/15411.$ 
This gives $B = \mat{731}{6224}{1810}{15411}$.

We have 

\begin{gather}
 B(1 + i) =  \frac{121095165}{299838941}+\frac{i}{299838941}\\
 B(1.0001 + 1.0001i) = 0.403867+ 3.335 \times 10^{-9}i \label{small}
 \end{gather}
 
 The ratio of the differences imaginary parts  and  real parts of the $B$-images is 
0.649359 which is within $\delta$ of   $\tan{33}^{\circ}  = 0.6494075932\ldots$ . And equation (\ref{small}) shows we are at height $3.335\ldots \times 10^{-9}$ when achieving this approximation.

\section{Homotopy classes of $I_\alp$, $\alp \geq 1/(2\sqrt{3})$.}\label{homo} 

In previous sections we have seen that taken as a set, the horocyclic paths $I_\alp, \: \alp >0$ can display recurrent behavior.  This naturally led us to study the homotopy classes of such curves. Here, two curves $f(t)$ and $g(t), \: t \in [0,1]$ on \modsurf \ are homotopic if there is  a continuous function $H$  from  $[0,1]^2$ to \modsurf \ with $H(t,0) = f(t)$ and $H(t,1) = g(t)$. 

The  main result of this section is  that as $\alp \downarrow 0$ the homotopy classes are stable --- do not change --- between encounters with elliptic fixed points.  Also, we determine the homotopy classes of $I_\alp$ for all $\alp$ in $ [1/(2\sqrt{3}), \infty).$  These classes become increasingly complicated as $\alp$ descends.

Not surprisingly we shall need to characterize the \modu \ elliptic fixed points at the outset.

\subsection{Some Lemmata: \modu \,  Images of $i$ and $\rho$} \label{lemmata}

Most but not all of the material in this subsection is in \cite{bS}, p.10-12.
We require the following computations.

Let $A \in \modu$ \, be $\mat{a}{b}{c}{d}$. We want to investigate  the images of $i$ and $\rho$ that lie on $I_{1/n}$.  First of all, we have

\begin{equation} \label{iimage}  A(i) = \frac{(ac + bd) + i}{c^2 + d^2} 
\end{equation} 

This means that $n = c^2 + d^2$, which is possible if and only if no prime congruent to 3 mod 4 appears in the factorization of $n$ (this is because $(c,d) = 1$).  Note also that given $n$, $c$ and $d$ are unique up to sign.  But as we shall see below, order matters.

What of $a$ and $b$?  We can use the euclidean algorithm to generate a minimal solution.  Call it again $a$ and $b$.  Then $a + mc$ and $b + md$ is also a solution (call it $A_m$) and every solution is easily shown the be of that form.  Here $m$ is any integer. 

Substituting this into equation (\ref{iimage}) gives $A_m(i) = A(i) + m$.  We have shown:

\begin{Lem} \label{lowi} The portion of $I_{1/n}$ with real part in $[0,1]$ contains an image of $i$ according as $n$ is a proper sum of two squares or not. There may be more than one image on such $I_{1/n}$.
\end{Lem}

The last statement is illustrated by the fact that both $\frac{2 + i}{5}$ and $\frac{3 + i}{5}$ are images of $i$ on $I_{1/5}$.  To wit, $\mat{1}{1}{1}{2} (i) = \frac{3 + i}{5}$ and  $\mat{1}{0}{2}{1} (i) = \frac{2 + i}{5}$.  There are more involved examples, with more than two  same-height $i$-images encountered in a (euclidean) length one segment of $I_{1/n}$. This can be realized exploiting the fact that $7^2 + 6^2 = 9^2 + 1^2 =85$. Last, it must also noted that $I_{1/5}$ arc between these two $1/5$-height images cannot be a fundamental arc for $I_{1/5}$.  If that were so, $z \mapsto z + 1/5$ would be in \modu.

Here is what happens with $\rho$:

\begin{equation} \label{rhoimage}  A(\rho) = \frac{2(ac + bd) + ad + bc + i\sqrt{3}}{2(c^2 + cd + d^2)} 
\end{equation}

\begin{Lem} The portion of $I_{1/n}$ with real part in $[0,1]$ contains an  image of $\rho$ according as whether or not $n$ is representable as $2(c^2 + cd + d^2), \, (c,d) =1$. There may be more than one image on such $I_{1/n}$.

\end{Lem}

As with Lemma ~\ref{lowi} the last statement is illustrated by the fact that both $\frac{9 + i\sqrt{3}}{14}$ and $\frac{5 + i\sqrt{3}}{14}$ are images of $\rho$ on $I_{1/14}$.  To wit, $\mat{1}{1}{1}{2} (\rho) = \frac{9 + i\sqrt{3}}{14} $ and  $\mat{1}{0}{2}{1} (\rho) =  \frac{5 + i\sqrt{3}}{14}$. The identical fundamental arc argument applies as well --- there may be many more $\rho$ images on $I_{1/n}$.

According to Leveque,\cite{wL} V. 2, Th 1-5, p.19, if there is a representation $n  = c^2 + cd + d^2$, then there is an $m$ with $0 \leq m < 2n$ and $4n | (m^2 + 3)$.  This is since 3 is the discriminant of  $ c^2 + cd + d^2$.  This means $m$ must be odd and writing $m= 2k+1$ we arrive at $n | k^2 + k + 1$.

\subsection{Closed Geodesics versus Closed Horocycles.}

Before giving our results, we foreground some differences between the homotopy classes for closed horocycles with such classes for closed geodesics.

\begin{itemize}

\item   First, there is but one geodesic in each non-trivial homotopy class --- lending rigidity to that study. For closed horocycles there are also, in the main, a continuum of curves in each homotopy class, however none is in any way distinguished.  

\item  Second, there need not be a geodesic in said homotopy class. (The curve collapses when pulled tight.) 

\item  Third, as we saw in section \ref{Versus},  the side pairings product has little geometric significance. In the geodesic case, this product gives the primitive hyperbolic whose fixed h-line projects to the geodesic. 

\item Fourth, as $I_\alp$ descends toward the real axis, it encounters efps.  These encounters separate homotopy classes --- that is, the set $\{z \in I_\alp | \, 0< \alp_0 < \alp < \alp_1< \infty \}$ is free of efps, if and only if its horocycles all have the same homotopy class. What's going on here is that the (accessible, non-cusp) vertices of SFR are efps.  Thus the $I_\alp$ which do not encounter an efp cannot change homotopy class as the lifts change continuously. When a lift arrives at an efp, there is an obvious change of homotopy class (to account for the incidence.)

\end{itemize}

\subsection{Closed Horocycle Homotopy Classes}\label{classes}

Consider all the $I_\alp$, where $\alp \in   (0, \infty)$.   In this section we examine the homotopy classes of the lifts of these closed curves to \modsurf \ for positive $\alp \geq 1/(2\sqrt{3})$. 

Our results are summarized in Figure 4 following the Bibliography,  which the reader will find helpful in reading the text of this section.  This color figure has 3 columns: 

\begin{itemize}

\item The first gives a range for $\alp$,  covering in total the interval $[1/(2\sqrt{3}), \infty]$. 

\item The second gives the lift of $I_\alp$ to \modsurf, which is a sphere with location indicated --- in the component of $\modsurf - \pi(I_\alp)$ or on $\pi(I_\alp)$ --- for the special points $i$, the efp2; $\rho$, the efp3; and $\times$, the cusp.   

\item The third gives the corresponding arc ensemble in the AFR, \emph{the alternate fundamental region} for \modu. AFR consists of  two half pieces of SFR translates lying between real parts 0 and 1.  Thus $\rho$ is the lowest point in the middle.  the side pairings are $T$ and $S T^{-1} : z \mapsto -1/(z-1)$ which takes $1 + i$ to $i$, fixes $\rho$ and rotates counter-clockwise by $2\pi/3$. 

The sequence (numbered) and direction (arrows)  as we enter and leave the AFR.  In some cases the pairing transformations in \modu \ are give as well as the entry/leaving points on the boundary of the AFR. Similarly, the identically colored regions are images of each other using the projection of \modsurf \ to $\mathcal{H}$.

\end{itemize}

 As we descend from $\infty$ to $I_1$, the homotopy class of the horocycle doesn't change:  it is a loop about $\infty$ with nothing else in component of the loop containing $\infty$.  (Obviously there is no geodesic in this homotopy class.) 
The fact that there is no change in the homotopy class as we descend to 1 may be seen directly.  However, and this will apply more generally as we descend without encountering efps, constructing a homotopy between two nearby closed horocycle paths is simply done using the small euclidean rectangle having said paths as horizontal edges.  The interior of this rectangle is contractable to either horocycle path in \modsurf.

When we hit $I_1$ something different happens:  we have a loop about $\infty$ beginning and ending at $i$ with a zero angle.

When we descend past $I_1$ we get a loop about $\infty$, which continues and ends  with an interior loop about $\rho$, interior meaning contained in  the component with $\infty$ in the previously produced loop. This homotopy class continues until we reach $\alp = \sqrt{3}/2$ where we hit $\rho$ (the interior loop is collapsing).   That gives us a loop about $\infty$, now ending at $\rho$ and with a (non-zero) angle of $\pi/3$.  At this juncture it  makes sense to look not at the standard SFR, but rather the two `halves' lying between real parts 0 and 1.  Thus $\rho$ is the lowest point in the middle.  the side pairings are $T$ and $S T^{-1} : z \mapsto -1/(z-1)$ which takes $1 + i$ to $i$, fixes $\rho$ and rotates counter-clockwise by $2\pi/3$. Let us denote this region AFR, for alternate fundamental region.

Upon descending from $\rho$ the curve is a figure 8, the intersection point is an ordinary one on the surface; the ``top" component contains $\infty$ while the bottom contains $\rho$.  The efp2 $i$ is in the exterior. 

Continuing down towards the real axis, we next get to $I_{1/2}$ which contains a \modu \ image of $i$, namely $(1 + i)/2 = S(-1 + i)$. Since here is no image of $\rho$ or $i$ at an intermediate height\footnote{Indeed the next lowest image if $\rho$ has imaginary part  is $1/(2\sqrt{3}) = 0.288675\ldots$; while the next lowest image of $i$ is at .2.} we get the next homotopy class change here:  the lower component of the figure 8 now has $i$ on its boundary.  

Proceeding down from $I_{1/2}$ yields a curve created from a figure 8 as follows: Place $i$ in the upper component.  Then put and internal loop in the lower component with the cusp inside the internal loop; $\rho$ is external to the whole figure, which has two self-intersections. This is homotopy class for $1/(2\sqrt{3}) \alp < 1/2$.  Of course $I_{1/(2\sqrt{3})}$ simply has $\rho$ on the boundary of the upper loop --- the one containing $i$ --- which is expanding toward $\rho$ as $\alp$ decreases.   

There is no impediment to continuing to decrease $\alp$ in this manner, though the computational aspects are more involve as we encounter further efps.  Since we have at this time no compelling argument for detailed study of these curves --- the pattern having been discerned already --- we stop at $I_{1/(2\sqrt{3})}$.

A word should be offered about the case of encountering two or more efps at  the same moment during the descent of $I_\alp$.  This means the curve encounters the same efp in its complete excursion more than once; it has a  self-intersection there.

Following the bibliography we offer a table giving the first eight homotopy classes together with their lifts to AFR and lifting transformations.  The most interesting of these is the class $\sqrt{3}/2 < \alp < 1 $.  The unusual feature being that, after removal of a simple horocyclic loop, two disjoint  disc components of \modsurf \  remain;  $\rho$ and $\infty$ are in one and $i$ is in the other.  

Here is a lemma useful in verifying the various lifts depicted in the Figure 4.  Its relevance stems from the fact that the $I_\alp$ don't get ``close" to the cusp (high in SFR).

\begin{Lem} Let  $J_\alp := ST^{-1}S(I_\alp)$ has the same euclidean radius as $S(I_\alp)$. Therefore, likewise for $K_\alp := TST^{-1}S(I_\alp) = T(J_\alp)$. The anchors of $J_\alp$ and $K_ \alp$ are $1$ and $2$, resp.
\end{Lem}

$ST^{-1}S = \mat{1}{0}{1}{1}$. Thus $J_\alp$ has anchor at 1 and $K_\alp$ has anchor 2.  The radial equalities also follow from this matrix expression.

\begin{figure} \label{Fi:newhoros}
\centering \includegraphics[scale=.24]{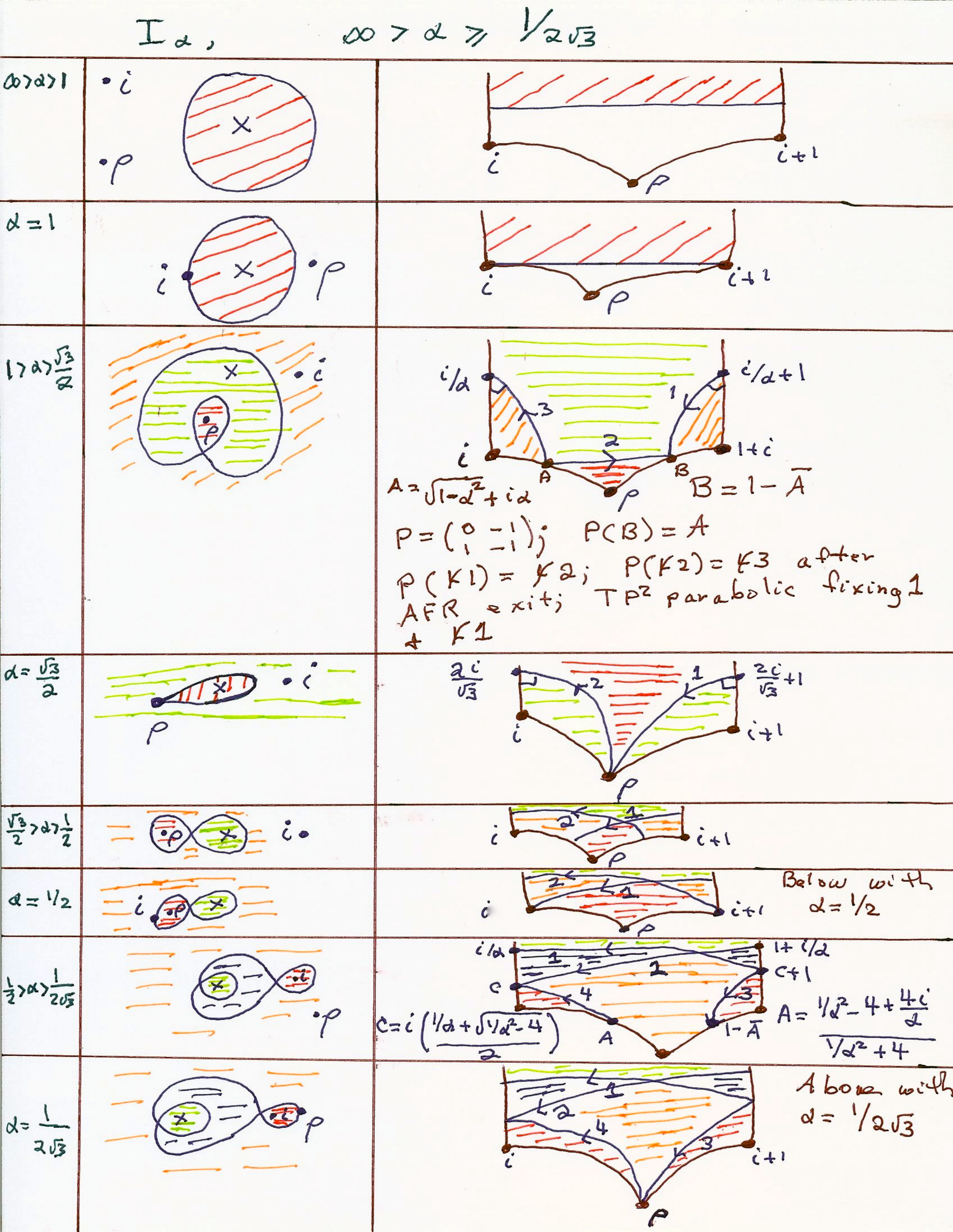}
\caption{$I_\alp, \quad  \infty >  \alp \geq  1/(2\sqrt{3})$}
\end{figure}

\end{document}